\documentclass[11pt]{article}
\usepackage{amsmath,amsfonts,amsthm,amssymb,mathtools}
\usepackage[linesnumbered,ruled]{algorithm2e}

\newcommand{\ceil}[1]{\lceil #1 \rceil}
\newtheorem{lem}{Lemma}

\newcommand{\floor}[1]{\lfloor #1 \rfloor}

\begin{document}

\title{Calculating the floor of $y^{1/m}$}

\author{Alexandros V. Gerbessiotis\thanks{CS Department, New Jersey Institute of Technology, 
Newark, NJ 07102, USA. Email: alexg@njit.edu}
}
\maketitle
\thispagestyle{empty}


\begin{abstract}
We present two algorithms based on the Newton-Raphson method
to calculate $\lfloor y^{1/m} \rfloor$ for natural integer numbers
$y>2$ and $m >1$. One could use such an algorithm to establish
whether $y$ is an integer power of an integer in number theory
problems, even though binary search methods are traditionally
considered simpler to implement.
\end{abstract}


\section{Introduction}

In several number theory algorithms we would like to determine
whether natural number $y$ is a perfect power that is, there
exist two natural numbers $x$ and $m>1$ such that $y= x^m$.
The primality testing algorithms of Miller \cite{M76} \cite{R80} and
AKS primality testing \cite{AKS} utilize such an 
identification method in the introductory steps of those
algorithms.  
The traditional approach to find $x,m$ is to use binary search.
For $y=x^m $ and $m>1$ implies $y=x^m \geq 2^m$ and therefore
$m \leq \lg{y}$. Therefore one tests for each integer  $m$ 
between integer two and the integer closer to $\lg{y}$ 
(or the floor of it) whether there exists an integer
$x$ such that  $y=x^m$ by performing a  binary search of $x$ 
on the integer interval between two and $y$. The exponentiation
algorithm with  repeated doubling can calculate $x^m$ 
efficiently within  the efficiency requirements of those 
two algorithms.
Furthermore, one need  only search for prime numbers $m$ 
in the given range and by doing so one can  further optimize 
the running time of the approach. 
In general the running time for perfect power
identification contributes low order terms to the overall
running time either of Miller's primality testing algorithms,
or the AKS algorithm.
The binary search-based method is not computationally efficient 
in the bit model of computation,
where bit-based arithmetic computations are used to express the
arithmetic complexity of an algorithm. Yet it is quite practical.

It is known that other methods can be used to solve the perfect
power identification problem as implied in \cite{BS} (page 59).
For example, Newton-Raphson based
methods have been used to determine the $m$-th root $\sqrt[m]{y}$
of an arbitrary real number $y$.
We utilize such a method to calculate $\floor{y^{1/m}}$.
In fact we present two different algorithms: they differ only on
the way they derive an integer-based version of Newton-Raphson.
Having computed $Y=\floor{y^{1/m}}$ one needs only test whether
$Y^m = y$ to determine not only whether $y$ is a perfect power but
calculate $x=Y$.

Such techniques using Newton-Raphson for integers problems are
not new e.g. see \cite{BS}. Our intent is just to provide 
an archival reference to such methods rather than making 
claims on their novelty.

\section{First approach}

The first approach is a quite straightforward adaptation
of a real number Newton-Raphson method for calculating
    $  y^{\frac{1}{m}}  $.
It is modified into  an integer-based Newton-Raphson method 
for calculating
    $ \left\lfloor y^{\frac{1}{m}} \right\rfloor $.
We call this approach Algorithm~\ref{F1}.
It is subsequently modified and refined into what 
we call Algorithm~\ref{F2}.

\begin{lem}
Given a natural integer number $y>2$ and
a natural integer number $m>1$,
Algorithm~\ref{F1} 
determines a natural integer number $x$ 
such that
\[
 x  =    \left\lfloor y^{\frac{1}{m}} \right\rfloor .
\]
\end{lem}

\begin{proof}
$ $ \\ $ $
In the remainder, $\lg{y}$ denotes the logarithm of $y$ base two.
We are going to use the following inequalities related to
the floor and ceiling functions.
\begin{equation}
\label{pp1a}
x-1  <  \floor{x} \leq x \leq \ceil{x} < x+1 .
\end{equation}
\begin{equation}
\label{pp1b}
1+ \floor{x} \geq \ceil{x} .
\end{equation}
$ $ \\ $ $
The symbol $\floor{< 0}$ will denote a negative quantity. The symbol
$\ceil{>0}$ will denote a positive quantity.
In either case we are only interested in the sign, but not the
quantity itself.
The Newton-Raphson method is highlighted below. For more information
see for example \cite{Rudin}.
$ $ \\ $ $
{\bf (a) Newton-Raphson iterative method: an overview.}
$ $ \\ $ $
Given $f(x)$ in Eq.(\ref{pp2}) that follow,
\begin{equation}
\label{pp2}
  f(x) = y - x^m ,
\end{equation}
its solution for $f(x)=0$ is $x = y^{1/m}$. Let $Y= y^{1/m}$.
Moreover we derive that
$f^\prime (x) = - m x^{m-1}$.
Let us then use the Newton-Raphson iterative method, where
$y^{1/m} = Y$ implies $y = Y^m$ as follows.
\begin{eqnarray}
\label{pp3}
x_{i+1} &=& x_{i} - \frac{f(x_i ) }{ f^\prime (x_i )} \nonumber \\
        &=& x_i + \frac{y-x_i^m}{m \cdot x_i^{m-1}}
        = x_i + \frac{Y^m -x_i^m}{m \cdot x_i^{m-1}},
\end{eqnarray}
for $i \geq 0$. The initial condition
$x_0$ will be determined and described later.
$ $ \\ $ $
{\bf (b) Integer Newton-Raphson.}
$ $ \\ $ $
In order to obtain an integer approximation, we
modify Eq.(\ref{pp3}) so that we get integer values at various
iterations.
\begin{eqnarray}
x_{i+1}
&=&
\label{pp4}
    x_i + \left\lfloor \frac{y-x_i^m}{m \cdot x_i^{m-1}} \right\rfloor
=
    x_i + \left\lfloor \frac{Y^m -x_i^m}{m \cdot x_i^{m-1}} \right\rfloor .
\end{eqnarray}
The numerator and denominator of the fraction are integer numbers,
so the fraction becomes the integer quotient of a long integer division.
If $x_0$ the sequence generated of $x_0 , x_1 , \ldots $ is an
integer sequence.
$ $ \\ $ $
{\bf (c) Choice of the initial value $x_0$.}
$ $ \\ $ $
The integer recurrence relation  of Eq.(\ref{pp4})
provides an integer sequence $x_0 , x_1 , \ldots $
as long as $x_0$, the initial condition, is also an integer.
We first note the following.
\begin{eqnarray}
\label{pp5}
 2^{1+ \floor{\ceil{\lg{y}}/m}} &\geq& 2^{\ceil{\ceil{\lg{y}}/m}} \geq
2^{\lg{y}/m} = y^{1/m} = Y,
\end{eqnarray}
At the same time,
\begin{eqnarray}
\label{pp5a}
 2^{1+ \floor{\ceil{\lg{y}}/m}} &\leq& 2 \cdot 2^{\ceil{\lg{y}}/m} \leq
4 \cdot 2^{\lg{y}/m} =  4 y^{1/m} = 4Y,
\end{eqnarray}
and therefore consider choosing an $x_0$ such that
\begin{eqnarray}
\label{pp5b}
Y= y^{1/m} \leq x_0 \leq  2 \cdot 2^{\ceil{\lg{y}}/m} \leq 4 y^{1/m} =4Y
  \Leftrightarrow\\
\label{pp5c}
 \lfloor \lg{y} \rfloor \leq \lg{y}  \leq m \cdot \lg{x_0}
\leq  m + {\ceil{\lg{y}}} .
\end{eqnarray}
$ $ \\ $ $
Starting with $y$, finding $\ceil{\lg{y}}$ and $\floor{\lg{y}}$
is easy as it requires shift operations and can be done in
$O(\lg{y})$ shifts.  
Then by raising integer two to the $\floor{\lg{y}}$-th power
we can verify whether
$\floor{\lg{y}} = \lg{y}$ or not.
Then,
the integer range $[\lg{y}  ,  m + {\ceil{\lg{y}}} ]$
contains at least $m$ consecutive integers and therefore an integer
that is a multiple of $m$.  We set that integer multiple of $m$ to  
$m \cdot \lg{x_0}$.  By implication $\lg{x_0}$ and thus $x_0$ become 
integer.
Then, $x_0$ still satisfies  inequalities
 Eq.(\ref{pp5c}) and Eq.(\ref{pp5b}).
Therefore $x_0$ is easy to compute and is calculated the
way described.
$ $ \\ $ $
{\bf (d) Iteration $i=0 \Rightarrow i+1=1$.}
Consider an $x_0$ chosen as described in (c) above
with $x_0$  satisfying Eq.(\ref{pp5b}) and Eq.(\ref{pp5c})
that yields $x_0 \geq y^{1/m} =Y$
and consequently  $x_0^m -y \geq 0$ and $x_0^m - Y^m \geq 0$ 
or, 
\begin{eqnarray}
\label{pp6}
y- x_0^m &\leq& 0 \Leftrightarrow  Y^m - x_0^m  \leq 0
\end{eqnarray}
By Eq.(\ref{pp4}) because of Eq.(\ref{pp6}) for $i=0$ we obtain that
$x_1 = x_0 + \floor{ \leq   0}  \leq  x_0$ as follows.
\begin{eqnarray}
\label{pp6a}
x_{1}
&=&
    x_0 + \left\lfloor \frac{Y^m -x_0^m}{m \cdot x_0^{m-1}} \right\rfloor
=
    x_0 + \left\lfloor \frac{y-x_0^m}{m \cdot x_0^{m-1}} \right\rfloor
\leq x_0 .
\end{eqnarray}
Note that the numerator $Y^m -x_0^m$ in Eq.(\ref{pp6a}) is 
non-positive by way of $x_0 \geq y^{1/m} =Y$.
If it is zero, then $Y^m = x_0^m$ and  $x_1 = x_0$,
which imply $x_0 = Y = y^{1/m}$ and since $x_0$
is an integer we have $\lfloor y^{1/m} \rfloor = x_0$ and 
the calculation goes no further.
Otherwise,
subsequence $x_0 , x_1 $ 
is a decreasing sequence since $x_0 \geq x_1 $ and it is not
$x_0 = x_1$.
Furthermore, because $x_0  \geq Y $, we have the following,
\begin{eqnarray}
x_0^m - y   =
x_0^m - Y^m &=& ( x_0 - Y) ( x_0^{m-1} + x_{0}^{m-2}Y + \ldots + Y^{m-1})
\nonumber \\
          &\leq& (x_0 - Y ) m x_0^{m-1} ,
\end{eqnarray}
and because of this, Eq.(\ref{pp6a}) further yields
the following.
\begin{eqnarray}
\label{ppa}
x_{1  }
   & =  &
    x_0 + \left\lfloor
            \frac{y -x_0^m}{mx_0^{m-1}}
          \right\rfloor  \nonumber  \\
   & \geq  &
   x_0 + \left\lfloor
            \frac{(Y-x_0 )m x_0^{m-1}}{mx_0^{m-1}}
          \right\rfloor  \nonumber  \\
   &\geq&  x_0 + (Y-x_0 ) -1  \nonumber \\
   &\geq&  Y -1   \nonumber \\
   &=&  y^{1/m} -1   \nonumber \\
   &\geq&  \floor{y^{1/m}} -1   .
\end{eqnarray}
We have just proved that 
$x_1 \geq Y -1$ and then further derived that 
$x_{ 1 } \geq \floor{y^{1/m}} -1$.
$ $ \\ $ $
{\bf (e) Iteration $i=1 \Rightarrow i+1=2$.}
Consider once more Eq.(\ref{pp4}).
\begin{eqnarray}
\label{pp7}
x_{2  } &=&
    x_1 + \left\lfloor
               \frac{Y^m - x_1^m}{mx_1^{m-1}}
         \right\rfloor
   \geq
    x_1 + \frac{Y^m - x_1^m}{mx_1^{m-1}} -1
\end{eqnarray}
We distinguish two   subcases then: (e1) and (e2), that
can be generalized for any $i>0$ by induction.
$ $ \\ $ $
{\bf (e1) Iteration $i=1 \Rightarrow i+1=2$ and $x_{1} \geq Y$.}
$ $ \\ $ $
For $i=1$ and $i+1=2$ and assuming $x_1 \geq Y= y^{1/m}$
we consquently obtain after noting that
$x_1^{m-1-j} Y^{j} \leq x_1^{m-1}$, $0\leq j \leq m-1$, the following.
\begin{eqnarray}
\label{pp8}
x_1^m - y  =
x_1^m - Y^m &=& ( x_1 - Y) ( x_1^{m-1} + x_{1}^{m-2}Y + \ldots + Y^{m-1})
\nonumber \\
          &\leq& (x_1 - Y ) m x_1^{m-1} ,
\end{eqnarray}
i.e. $x_1^m - Y^m \leq (x_1 - Y) m x_1^{m-1}$.
Reversing the inequality we obtain
equivalently that
$Y^m - x_1^m \geq (Y- x_1 ) m x_1^{m-1}$.
Therefore, for $x_1 \geq Y$,  Eq.(\ref{pp7}) because of
Eq.(\ref{pp8}) gives the following.
\begin{eqnarray}
\label{pp9}
x_{2  }
   & =  &
    x_1 + \left\lfloor
                \frac{Y^m-x_1^m}{mx_1^{m-1}}
         \right\rfloor \nonumber    \\
   &\geq&
    x_1 + \left\lfloor
                \frac{(Y-x_1 )m x_1^{m-1}}{mx_1^{m-1}}
          \right\rfloor \nonumber  \\
   &\geq&  x_1 + (Y-x_1 ) -1  \nonumber \\
   &\geq&  Y -1   \nonumber \\
   & =  &  y^{1/m} -1   \nonumber \\
   &\geq&  \floor{y^{1/m}} -1   .
\end{eqnarray}
We have just proved that 
if $x_1 > Y = y^{1/m}$ then,
$x_2 \geq Y-1$ and also derived
$x_{ 2 } \geq \floor{y^{1/m}} -1$.
%
This generates the following lemma by induction.
\begin{lem}
\label{ppl1}
If $x_i \geq Y= y^{1/m}$ then $x_{i+1} \geq Y-1 = y^{1/m} -1$,
for each  $i \geq 0$.
\end{lem}
\bigskip
Note that $x_i \geq Y= y^{1/m}$ is checked by comparing
$x_i^m \geq y$ instead and the same applies for the condition
of  Lemma~\ref{ppl2}.
We have also proved the result summarized in Lemma~\ref{ppl2} below.
\begin{lem}
\label{ppl2}
If $x_i \geq Y= y^{1/m}$ then $x_{i+1} \leq x_i$, for each 
$i \geq 0$ 
Furthermore
if it is not $x_{i+1} < x_{i}$ then $x_{i+1} = x_{i}$ 
and $\lfloor y^{1/m} \rfloor = x_i$.
\end{lem}
\begin{proof}
This is a consequence of
Eq.(\ref{pp4}) which shows that if
$x_i \geq Y= y^{1/m}$, then  $y- x_i^m < 0$, and
$x_{i+1} = x_i + \floor{ \leq  0} \leq x_i$.
The latter part of Lemma~\ref{ppl2} is similarly derived from 
the discussion of part (d) earlier. If $x_{i+1} = x_{i}$ then
$Y^m = y =  x_i^m$ and  $y^{1/m} = \lfloor y^{1/m} \rfloor = x_i$.
\end{proof}
$ $ \\ $ $ 
We move to case (e2).
By Eq.(\ref{ppa}) $x_1 \geq Y-1$. If it is not case (e1)
where $x_1 \geq Y$, then the only case left is
$Y > x_1 \geq Y-1$. This leads to case (e2).
$ $ \\ $ $
{\bf (e2) Iteration $i=1 \Rightarrow i+1=2$ and $Y> x_{1} \geq Y-1$.}
$ $ \\ $ $
In this case $x_1^m < Y^m$ and thus
\begin{eqnarray}
\label{pp10}
y   - x_1^m  =
Y^m - x_1^m  &=& ( Y - x_1 )( x_1^{m-1} + x_{1}^{m-2}Y + \ldots + Y^{m-1})
\nonumber \\
          &\geq& (Y- x_1  ) m x_1^{m-1} ,
\end{eqnarray}
Subsequently, we have the following by way of Eq.(\ref{pp4}) for $i=1$.
\begin{eqnarray}
\label{pp10a}
x_{2  }
   & =  &
    x_1 + \left\lfloor
                \frac{y-x_1^m}{mx_1^{m-1}}
         \right\rfloor \nonumber    \\
   & =  &
    x_1 + \left\lfloor
                \frac{Y^m-x_1^m}{mx_1^{m-1}}
         \right\rfloor \nonumber    \\
   &\geq&
    x_1 + \left\lfloor
                \frac{(Y-x_1 )m x_1^{m-1}}{mx_1^{m-1}}
          \right\rfloor \nonumber  \\
   &\geq&  x_1 + (Y-x_1 ) -1  \nonumber \\
   &\geq&  Y -1   \nonumber \\
   & =  &  y^{1/m} -1   \nonumber \\
   &\geq&  \floor{y^{1/m}} -1   ,
\end{eqnarray}
as in subcase (e1).
The corresponding Lemmas are then as follows.
\begin{lem}
\label{ppl1a}
If $Y> x_i \geq Y-1= y^{1/m}-1$ then $x_{i+1} \geq Y-1 = y^{1/m} -1$,
for each  $i \geq 0$.
\end{lem}
\begin{lem}
\label{ppl2a}
If $Y> x_i \geq Y-1= y^{1/m}-1$ then $x_{i+1} \geq x_i$,
for each  $i \geq 0$.
Furthermore
if it is not $x_{i+1} > x_{i}$ then $x_{i+1} = x_{i}$ 
and $\lfloor y^{1/m} \rfloor = x_i$.
\end{lem}
The proof of Lemma~\ref{ppl1a} follows from Eq.(\ref{pp10a}).
The proof of Lemma~\ref{ppl2a} follows from Eq.(\ref{pp4})
and a discussion similar to the proof of Lemma~\ref{ppl2}.
$ $ \\ $ $
We wrap things up.
Consider the sequence of values $x_0 , x_1 , x_2 , \ldots$
obtained by applying the
Newton-Raphson formula Eq.(\ref{pp4}).
One of three cases might apply.
$ $ \\ $ $
{\bf Case 1.}
Let $x_{n+1} = x_n$ occurs for the first time after a decreasing
sequence i.e. we have $x_0 > x_1 > \ldots > x_{n-1} > x_n = x_{n+1}$.
Note that for a decreasing sequence
 $x_{i+1} =x_i + \floor{<0} \leq x_i -1 < x_i$.
In order to have
$x_{n+1} = x_n$ we must have that the quantity within the floor
of  Eq.(\ref{pp4}) is non-negative,
i.e. $x_n$ is such that $y- x_n^m \geq 0$, thus implying
$x_n^m \leq y$ or equivalently
$x_n \leq  y^{1/m}$.  This is sub-case (e2) earlier.
If $x_n =y^{1/m}$ then $x_n = \floor{y^{1/m}}$ and the calculation
is over. Otherwise $x_n <y^{1/m}$.
This combined with the precondition of Lemma~\ref{ppl1}
(or Lemma~\ref{ppl2}) for
$i=n-1$ which leads to $x_{n-1} > y^{1/m}$ results in the following:
\[
x_{n-1} > y^{1/m} \geq x_n.
\]
Note that if $x_{n-1} = Y$ then $x_n = x_{n-1}$ but
we have $x_{n-1}  > x_{n}$.
In addition, by  Lemma~\ref{ppl2} for $i=n-1$ we know that
$x_n \geq Y-1 \geq  \floor{y^{1/m}} -1$.
Combining the two we obtain the following.
\[
x_{n-1} > y^{1/m} \geq x_n  \geq \floor{y^{1/m}} -1.
\]
Therefore we can find $ \floor{y^{1/m}}$ by just considering one
of $x_n$ and $x_n +1$:
one of these terms is $ \floor{y^{1/m}}$.
$ $ \\ $ $
{\bf Case 2.}
After a decreasing sequence we have a flip that is
we have $x_0 > x_1 > \ldots > x_{n-1} > x_n < x_{n+1}$.
Note that by the base case and the choice of $x_0$
we have $x_0 \geq x_1 $, therefore the sequence is non-increasing;
in fact it is decreasing at the start unless $x_0 = y^{1/m} =Y$, 
$Y$ is integer, and then we can stop for $x_0 =\floor{y^{1/m}}$.
Therefore in case 2 we have that $x_{n-1} > x_n$ but $x_{n+1} > x_n$
results in a flip. The sequence flips for the first time. 
By Lemma~\ref{ppl1} for $i=n-1$  we have that
$x_{n-1} \geq Y=y^{1/m}$ and
$x_{n} \geq Y-1$.
By Lemma~\ref{ppl2} for $i=n-1$  we have that
$x_{n-1} \geq Y=y^{1/m}$ and
$x_{n} < x_{n-1}$. This is because the sequence is decreasing
up to that point.
For $x_n$ we have obtained from above that $x_n \geq Y-1$.
There are two possibilities for $x_n$: case (2a) where $x_n \geq Y$
and case (2b) where $x_n < Y$.
If case (2a) is applicable, then Lemma~\ref{ppl1} implies
$x_{n+1} \leq x_n$. However at $n+1$ we have a flip
with $x_{n+1} > x_n$. Therefore case (2a) is impossible.
We are left with case (2b) and $x_n $ such that
$Y > x_n > Y-1$. Then Lemma~\ref{ppl1a} applies to
give $x_{n+1} \geq Y-1$, and Lemma~\ref{ppl2a}
to derive $x_{n+1} \geq x_n$. But since $x_{n+1} > x_{n}$
equality $x_{n+1}=x_n$  is not possible and 
$x_{n+1} > x_n$ is confirmed
If we combine the previous derivations we obtain
\[
x_{n-1} \geq Y= y^{1/m} >    x_n  \geq Y-1 = \floor{y^{1/m}} -1.
\]
Then, $ \floor{y^{1/m}}$ is one of $x_n$ or $x_n +1$.
$ $ \\ $ $
{\bf Case 3.}
If $x_{n+1}<x_n$, then continue with the next iteration.
$ $ \\ $ $
{\bf Conclusion.}
$ $ \\ $ $
In either {    Case 1} or {    Case 2}
we stop at an iteration where either the sequence flips sign
(the next term become larger than the previous term)
or the next term is equal to the previous one.
If the next
term is smaller than the one before ({    Case 3})
we continue to generate the next term of the  sequence.
As soon as we generate the flipping term ($x_{n+1}$ in Case 2)
or the repeating term ($x_{n+1} = x_n$ in Case 1) we stop
Newton-Raphson, and go back to $x_n$. One of
$x_n$ or $x_n +1$ is 
$ \floor{y^{1/m}}$. We determine the answer
by calculating $x_n^{m}$ and $(x_n +1)^m$, as needed.
In other words, in order to determine whether $x_n = \floor{y^{1/m}}$,
we assume that if it is true
\begin{equation}
\label{eqFloor}
x_n \leq  y^{1/m} < x_n +1
\Leftrightarrow
x_n^m \leq  \left( y^{1/m}\right)^m < \left( x_n +1 \right)^m
\Leftrightarrow
x_n^m \leq  y  < \left( x_n +1 \right)^m ,
\end{equation}
and thus we check the last part of Eq.(\ref{eqFloor}).
We do likewise to determine whether  $x_n +1$ is a candidate for 
$x_n + 1 = \floor{y^{1/m}}$.
$ $ \\ $ $
We finally note that if the calculation of $\floor{y^{1/m}}$
relates to determining whether $y$
is an integer power of an integer, we just need to
verify whether $x_n^m$ of $(x_n +1)^m$ is equal to $y$ or not
and thus the checking of Eq.(\ref{eqFloor}) can be skipped.
$ $ \\ $ $

\SetKwComment{Comment}{/* }{ */}
\SetKwRepeat{Do}{do}{while}
\begin{algorithm}[H]
\KwIn{Natural number $y$ greater than 2, natural number $m >1$}
\KwOut{Integer $x$ such that $\lfloor x^{1/m} \rfloor = y$}

 $i=0$;
 $x_0 = Init(y,m)$  
\Comment*[r]{Init(y,m) Per Eq.(\ref{pp5b}) and Eq.(\ref{pp5c})}

\While{True}{

 $x_{i+1} =
    x_i + \left\lfloor \frac{y-x_i^m}{m \cdot x_i^{m-1}} \right\rfloor $;

 \If(\tcc*[f]{Case 1}){$x_{i+1} = x_{i}$}{
  \If{$\mathbf{CheckCandidateSolution} ( x_{i},y,m)    = \mathbf{YES}$}{
     \Return{$x_{i}$};
  }
  \If{$\mathbf{CheckCandidateSolution} ( x_{i}+1,y,m)  = \mathbf{YES}$}{
     \Return{$x_{i}$};
  }
 }
 \If(\tcc*[f]{Case 2}){$x_{i-1} > x_{i}$ and $x_{i} < x_{i+1}$}{
  \If{$\mathbf{CheckCandidateSolution} ( x_{i},y,m)   = \mathbf{YES}$}{
     \Return{$x_{i}$};
  }
  \If{$\mathbf{CheckCandidateSolution} ( x_{i}+1,y,m) = \mathbf{YES}$}{
     \Return{$x_{i}$};
  }
 }

 \If(\tcc*[f]{Case 3}){$x_{i+1} < x_{i}$}{
  $i=i+1;$

  continue \Comment*[r]{Move to the next iteration}
 }

}

 \caption{FindFloor(y,m) : Find $\lfloor y^{1/m} \rfloor$}
 \label{F1}
\end{algorithm}
\end{proof}

$ $ \\ $ $
{\bf Example 1.}
For the simple case $y=10^8$ and $m=8$,
the sequence generated is $16, 14, 12$, $10, 10$ and
$x=10$ is reported.

$ $ \\ $ $
{\bf Example 2.}
For the case $y=96889010407$ and $m=13$,
the sequence generated is $8,7,7$ 
and $x=7$ is reported.

$ $ \\ $ $
{\bf Example 3.}
For the case $y=52523350144$ and $m=7$,
the sequence generated is $32,34,34$
and $x=34$ is reported.

$ $ \\ $ $
{\bf Example 4.}
For the case $y=52523350141$ and $m=7$,
the sequence generated is $32,34$
and $x=33$ is reported.

\section{A second approach}

Because of the case complexity involved in the analysis
of Algorithm~\ref{F1} we pursued another interpretation of
Integer Newton-Raphson for $\floor{y^{1/m}}$.
We call this second approach Algorithm~\ref{F2}.

\begin{lem}
Given a natural integer number $y$ and
a natural integer number $m>1$ Algorithm~\ref{F2}
determines a natural integer number $x$ such that
\[
    x = \left\lfloor y^{\frac{1}{m}} \right\rfloor .
\]
\end{lem}

The proposed solution uses a different sequence of
floor operations in converting the generic Newton-Raphson
of Eq.(\ref{pp3}) into Eq.(\ref{pp4}) that was
used in  Algorithm~\ref{F1}.

\begin{proof}
$ $ \\ $ $
We are going to use a nested inequality of floor functions
in addition to the ones established with the previous problem.
We quote~\cite{Wiki26a}, for integer $n,m$ (in our case positive
integer) and real (in our case positive) $x$.
\begin{equation}
\label{nnp1}
\lfloor \frac{\floor{x} + m}{n} \rfloor
=
\lfloor \frac{x + m}{n} \rfloor
\end{equation}
$ $ \\ $ $
We proceed to using the  Newton-Raphson approximation as follows. \\
$ $ \\ $ $
{\bf (a) Newton-Raphson iterative method.}
$ $ \\ $ $
The solution for $f(x)=0$ of Eq.(\ref{pp2}) as before is
is $x = y^{1/m}$. Moreover we derive that
$f^\prime (x) = - m x^{m-1}$.
Let us then use the Newton-Raphson iterative method, where
$y^{1/m} = Y$ implies $y = Y^m$ as follows.
\begin{eqnarray}
\label{npp3}
x_{i+1} &=& x_{i} - \frac{f(x_i ) }{ f^\prime (x_i )}
        = x_i + \frac{y-x_i^m}{m \cdot x_i^{m-1}}
        = x_i - \frac{x_i}{m} + \frac{y}{m \cdot x_i^{m-1}}  \nonumber \\
        &=& \left( (m-1) x_i  +  \frac{y}{x_i^{m-1}} \right)
            \cdot  \frac{1}{m},
\end{eqnarray}
for $i \geq 0$.  The initial condition  $x_0$ will be determined 
and described later.
$ $ \\ $ $
{\bf (b) Integer Newton-Raphson.}
$ $ \\ $ $
In order to obtain an integer approximation, we
modify Eq.(\ref{npp3}) so that we get integer values at various
iterations.
\begin{eqnarray}
x_{i+1}
&=&
\label{npp4a}
  \left\lfloor
     \left( (m-1) x_i + \left\lfloor \frac{y}{x_i^{m-1}} \right\rfloor
     \right) \cdot  \frac{1}{m}
  \right\rfloor ,
\end{eqnarray}
which by way of Eq.~(\ref{nnp1}) is equivalent to
\begin{eqnarray}
x_{i+1}
&=&
\label{npp4b}
  \left\lfloor
     \left( (m-1) x_i +   \frac{y}{x_i^{m-1}}
     \right) \cdot  \frac{1}{m}
  \right\rfloor .
\end{eqnarray}
The numerator and denominator of the internal fraction are
integers.
The numerator external fraction in Eq.(\ref{npp4a}) is an
integer. The denominator is consider to be the $m$ of $1/m$,
which is also an integer. Algorithm~\ref{F2} uses Eq.(\ref{npp4a}),
though its analysis would also use Eq.(\ref{npp4b}).
$ $ \\ $ $
{\bf (c) Choice of the initial value $x_0$.}
$ $ \\ $ $
This integer recurrence relation of Eq.(\ref{npp4a}) provides 
an integer sequence $x_0 , x_1 , \ldots $ as long as 
$x_0$, the initial condition, is also an integer.
We first note the following.

We choose $x_0$ such that
\begin{eqnarray}
\label{npp5}
 x_0 =      2^{\ceil{  \frac{\floor{\lg{y}}+1}{m} }  }
     &\geq& 2^{        \frac{\floor{\lg{y}}+1}{m}    }
      \geq  2^{        \frac{\lg{y}}{m}              }
       =    y^{1/m} = Y \geq \floor{y^{1/m}}.
\end{eqnarray}
i.e. $ x_0 \geq y^{1/m} = Y \geq \floor{y^{1/m}}$.
Starting with $y$, finding $\ceil{\lg{y}}$ and $\floor{\lg{y}}$
is easy as it requires shift operations and can be done in
$O(\lg{y})$ shifts.  Therefore $x_0$ is easy to compute
and is calculated the way described.
$ $ \\ $ $
{\bf (c) Minimization of Newton-Raphson term of Eq.(\ref{npp4b}).}
$ $ \\ $ $
From Eq.(\ref{npp4b}) that is equivalent to Eq.(\ref{npp4a}) 
we proceed as follows by isolating the argument of the floor
function.
\begin{eqnarray}
 f(x)
&=&
\label{npp4c}
     \left( (m-1) x +   \frac{y}{x^{m-1}}
     \right) \cdot  \frac{1}{m}
\end{eqnarray}
thus transforming equation Eq.(\ref{npp4b}) into the following
form.
\begin{eqnarray}
x_{i+1}
&=&
\label{npp4d}
  \left\lfloor
   f(x_i )
  \right\rfloor .
\end{eqnarray}
We observe the following
\[
 f^\prime (x) = \left( 1 - \frac{1}{m} \right)
                -
                y \cdot \left(  1 - \frac{1}{m} \right) x^{-m} .
\]
For
\[
f^\prime (x) = 0 ,
\]
we obtain (assuming $m \neq 1$, which is the case)
\[
  x= Y= y^{\frac{1}{m}} .
\]
Furthermore,
\[
f^{''} (x) = y m \left(  1 - \frac{1}{m} \right) x^{-m+1},
\]
and setting $  x= y^{\frac{1}{m}}$ we
obtain
\[
f^{''} (y^{\frac{1}{m}}) =
    y m \left(  1 - \frac{1}{m} \right) \cdot \frac{1}{y}
                    \cdot \frac{1}{y^{\frac{1}{m}}}
\]
which is positive for $1-1/m>0$, which is the case
since $m>1$.
Therefore function $f(x)$ has a minimum at
$x= y^{\frac{1}{m}}$
and calculating that minimum
\[
f(x)\left|_{x= y^{\frac{1}{m}}} \right. = y^{\frac{1}{m}} .
\]
Therefore from Eq.(\ref{npp4c}) and Eq.(\ref{npp4d}
we obtain the following.
\begin{eqnarray}
x_{i+1}
&=&
  \left\lfloor
     \left( (m-1) x_i +   \left\lfloor
                          \frac{y}{x_i^{m-1}}
                          \right\rfloor
     \right) \cdot  \frac{1}{m}
  \right\rfloor \nonumber \\
&=&
\label{npp5a}
  \left\lfloor
    f(x_i )
  \right\rfloor  \geq \floor{ y^{\frac{1}{m}} }.
\end{eqnarray}

\noindent
We obtain the first of three lemmas to follow.
\begin{lem}
\label{lemn1}
If $x_0 \geq y^{\frac{1}{m}} \geq
\floor{ y^{\frac{1}{m}} }$,
then $x_i \geq \floor{ y^{\frac{1}{m}} }$.
\end{lem}
We note that by Eq.(\ref{npp5}) it is 
indeed $x_0 \geq y^{\frac{1}{m}}$,
and also $x_0 \geq  \floor{y^{\frac{1}{m}}}$.

$ $ \\ $ $
{\bf (d) Sequence $x_0 , x_1 , x_2 , \ldots $.}
$ $ \\ $ $
From Eq.(\ref{npp4a}) or equivalently Eq.(\ref{npp4b})
we obtain the following.
\begin{eqnarray}
\label{npp6ab}
x_{i+1}
&=&
  \left\lfloor
     \left( (m-1) x_i +   \left\lfloor
                          \frac{y}{x_i^{m-1}}
                          \right\rfloor
     \right) \cdot  \frac{1}{m}
  \right\rfloor  \\
&=&
\label{npp6ba}
  \left\lfloor
     \left( (m-1) x_i +
                          \frac{y}{x_i^{m-1}}
     \right) \cdot  \frac{1}{m}
  \right\rfloor .
\end{eqnarray}
Consider the case $i=0$ and thus $i+1 =1$.
\begin{eqnarray}
x_{1}
&=&
\label{npp6b}
  \left\lfloor
     \left( (m-1) x_0 +
                          \frac{y}{x_0^{m-1}}
     \right) \cdot  \frac{1}{m}
  \right\rfloor
\end{eqnarray}
From Eq.(\ref{npp5}) we have $x_0 \geq y^{1/m}$
and therefore $y/ x_0^{m-1} \leq x_0$.
Eq.(\ref{npp6b}) then yields the following.
\begin{eqnarray}
x_{1}
=
\label{npp6c}
  \left\lfloor
     \left( (m-1) x_0 +
                          \frac{y}{x_0^{m-1}}
     \right) \cdot  \frac{1}{m}
  \right\rfloor
   &\leq &
  \left\lfloor
     \left( \left( m-1 \right) x_0 +  x_0
     \right) \cdot  \frac{1}{m}
  \right\rfloor   = \floor{ x_0 } \leq x_0.
\end{eqnarray}
We could use the same argument to show that $x_2 \leq x_1 $.
However we need
$x_1 \geq y^{1/m}$ but by Lemma~\ref{lemn1}
we only have
$x_1 \geq \floor{ y^{\frac{1}{m}} }$.
The only case that $x_1 \geq y^{1/m}$ is NOT the case
is for $ \floor{ y^{\frac{1}{m}} } = y^{1/m}$ but then
we stop the Newton-Raphson calculation .
We thus get the following second lemma.

\begin{lem}
\label{lemn2}
If $x_0 \geq y^{\frac{1}{m}} \geq
\floor{ y^{\frac{1}{m}} }$,
then $x_i \geq  y^{\frac{1}{m}} $.
Moreover if
$x_i  >    y^{\frac{1}{m}} $,
then $x_{i+1} < x_i$.
\end{lem}
\begin{proof}
The first part is Lemma~\ref{lemn1}.
For the second part, if $x_i >  y^{\frac{1}{m}} $ we
utilize Eq.(\ref{npp6ab}) and similarly to the derivation
$x_1 \leq x_0$ earlier, we derive first $ x_i^m > y $
and then $x_i > a= y / x_i^{m-1}$.
\begin{eqnarray}
\label{npp6d}
x_{i+1}
&=&
  \left\lfloor
     \left( (m-1) x_i +   \left\lfloor
                          \frac{y}{x_i^{m-1}}
                          \right\rfloor
     \right) \cdot  \frac{1}{m}
  \right\rfloor  \nonumber \\
&=&
  \left\lfloor
     \left( (m-1) x_i +   \left\lfloor
                          a
                          \right\rfloor
     \right) \cdot  \frac{1}{m}
  \right\rfloor  \nonumber \\
&\leq&
  \left\lfloor
     \left( (m-1) x_i +  
                          x_i -1
     \right) \cdot  \frac{1}{m}
  \right\rfloor  \nonumber \\
&=&
\label{npp6e}
  \left\lfloor
     \left( m x_i - 1 
     \right) \cdot  \frac{1}{m}
  \right\rfloor \nonumber \\
&\leq& x_{i}-1 < x_i
\end{eqnarray}
\end{proof}

The following lemma covers the case $x_i = y^{\frac{1}{m}}$.
\begin{lem}
\label{lemn3}
If $x_0 \geq y^{\frac{1}{m}} \geq
\floor{ y^{\frac{1}{m}} }$,
then $x_i  \geq  y^{\frac{1}{m}} $.
If $x_i \leq y^{\frac{1}{m}}$, 
because
$x_i \geq \floor{ y^{\frac{1}{m}} }$ and $x_i$ is
integer,
we have $x_i = \floor{ y^{\frac{1}{m}} } = y^{\frac{1}{m}}$.
\end{lem}
\begin{proof}
The first part is Lemma~\ref{lemn1}.
If $x_i \geq  y^{\frac{1}{m}} $ and it cannot be
 $x_i  >    y^{\frac{1}{m}} $ the only possibility
left is
 $x_i  =    y^{\frac{1}{m}} $ rather than the
stated
 $x_i  \leq y^{\frac{1}{m}} $.
Since $x_i$ is an integer it follows that
$x_i = y^{\frac{1}{m}} =\floor{y^{\frac{1}{m}}}$.
\end{proof}
$ $ \\ $ $
Consider the sequence of values $x_0 , x_1 , x_2 , \ldots$
obtained by applying the
Newton-Raphson formula Eq.(\ref{npp6ab}).
We distinguish three cases.
$ $ \\ $ $
{\bf Case 1.} If $x_{i+1} < x_i$ we continue.
$ $ \\ $ $
{\bf Case 2.}
If $x_{i+1}  \geq   x_i$, we stop.
$ $ \\ $ $
By Lemma~\ref{lemn1} we have $x_i \geq \floor{y^{\frac{1}{m}}}$.
By the stopping condition we have $x_{i+1} \geq x_i$ and by
combining the two we obtain the following.
\[
x_{i+1} \geq   x_i \geq \floor{y^{\frac{1}{m}}}.
\]
There are two possibilities:
(a) $ x_i \leq y^{\frac{1}{m}}$
or
(b) $x_i >y^{\frac{1}{m}}$.
Case (b) is impossible because it contradicts Lemma~\ref{lemn2}
that states it should be then $x_{i+1}<x_{i}$.
Case (a) is the only case possible.
We obtain
\[
y^{\frac{1}{m}} \geq x_i \geq \floor{y^{\frac{1}{m}}},
\]
utilizing Lemma~\ref{lemn1} . This implies
$x_i = y^{\frac{1}{m}} = \floor{y^{\frac{1}{m}}}$.
$ $ \\ $ $
{\bf Case 3.} We stop when $x_i = y^{1/m}$ or when $y^{1/m} > x_i$.
We test equivalently $x_i^m = y$ or $y > x_i^m$.
$ $ \\ $ $
By Eq.(\ref{npp5a}), we have
$x_{i+1} = \floor{f(x_i )}$.
We then claim that $f(x_i ) \geq x_i$. Suppose
this is not the case and we have instead
$f(x_i ) < x_i$. By Eq.(\ref{npp5a}),
$x_{i+1} = \floor{f(x_i )} \leq f(x_i ) < x_i$.
Then we have case 1 and for case 1 we continue iterating
than stopping. Therefore we must have $f(x_i ) \geq x_i$.
Then we obtain the following.
\begin{eqnarray}
 f(x_i ) & \geq & x_i
                  \Leftrightarrow \nonumber \\
     \left( (m-1) x_i +   \frac{y}{x_i^{m-1}}
     \right) \cdot  \frac{1}{m}
 &\geq& x_i       \Leftrightarrow \nonumber \\
    x_i^m &\leq& y^{\frac{1}{m}}
                  \Leftrightarrow \nonumber \\
\label{npp7}
    x_i &\leq& y^{\frac{1}{m}}.
\end{eqnarray}
From Eq.(\ref{npp7}) and combining with Lemma~\ref{lemn1} we have
the following.
\[
\floor{ y^{\frac{1}{m}}}+ 1
> y^{\frac{1}{m}} \geq x_i \geq \floor{ y^{\frac{1}{m}}} ,
\]
where the third
inequality is by way of Lemma~\ref{lemn1}, and
the second by the derivation above.
This implies $x_i = \floor{ y^{\frac{1}{m}}}$.
\[
\floor{ y^{\frac{1}{m}}}+ 1
> y^{\frac{1}{m}} \geq x_i \geq \floor{ y^{\frac{1}{m}}} ,
\]
that leads to $x_i = \floor{ y^{\frac{1}{m}}}$.
\end{proof}
$ $ \\  $ $
\SetKwComment{Comment}{/* }{ */}
\SetKwRepeat{Do}{do}{while}
\begin{algorithm}[H]
\KwIn{Natural number $y$ greater than 2, natural number $m >1$}
\KwOut{Integer $x$ such that $\lfloor x^{1/m} \rfloor = y$}

 $i=0$;
 $x_0 = Init2(y,m)$  
\Comment*[r]{Init2(y,m) Per Eq.(\ref{npp5}))}

\While{True}{

 $x_{i+1} =
  \left\lfloor
     \left( (m-1) x_i +   \left\lfloor
                          \frac{y}{x_i^{m-1}}
                          \right\rfloor
     \right) \cdot  \frac{1}{m}
  \right\rfloor  $ ;

 \If(\tcc*[f]{Case 1}){$x_{i+1} < x_{i}$}{
    $i=i+1;$

    continue \Comment*[r]{Move to the next iteration}
 }
 \If(\tcc*[f]{Case 2}){$x_{i+1} \geq x_{i}$}{
  \If{$\mathbf{CheckCandidateSolution} ( x_{i},y,m)   = \mathbf{YES}$}{
     \Return{$x_{i}$};
  }
 }

 \If(\tcc*[f]{Case 3}){$x_i^m = y$ \textbf{or} $y> x_i^m$}{
   \If{$\mathbf{CheckCandidateSolution} ( x_{i},y,m)   = \mathbf{YES}$}{
      \Return{$x_{i}$};
   }

 }
}
 \caption{FindFloor2(y,m) : Find $\lfloor y^{1/m} \rfloor$}
 \label{F2}
\end{algorithm}

\SetKwComment{Comment}{/* }{ */}
\SetKwRepeat{Do}{do}{while}
\begin{algorithm}[H]
\KwIn{Natural numbers $x,y,m $  greater than one}
\KwOut{YES if $\lfloor y^{1/m} \rfloor = x$; NO otherwise}

 \eIf{$x^m \leq y$ and $ y < (x+1)^m$}{
       \Return{$\mathbf{YES}$}
 }{
       \Return{$\mathbf{NO}$}
 }

 \caption{CheckCandidateSolution(x,y,m) : Is $x= \lfloor y^{1/m} \rfloor$?}
 \label{CF}
\end{algorithm}

$ $ \\ $ $
{\bf Example 5.}
For the simple case $y=10^8$ and $m=8$,
the sequence generated is $16,14,12,10,10$ and
$x=10$ is reported.

$ $ \\ $ $
{\bf Example 6.}
For the case $y=96889010407$ and $m=13$,
the sequence generated is $8,7,7$ 
and $x=7$ is reported.

$ $ \\ $ $
{\bf Example 7.}
For the case $y=52523350144$ and $m=7$,
the sequence generated is $64, 54, 46, 40, 36$,
$34, 34$
and $x=34$ is reported.

$ $ \\ $ $
{\bf Example 8.}
For the case $y=52523350141$ and $m=7$,
the sequence generated is $64, 54, 46, 40$, 
$36, 34, 33, 34$
and $x=33$ is reported.

\section{Conclusion}

If one compares the  performance of the two
algorithms through the series of examples indicated,
they show that Algorithm~\ref{F1} is
slightly better than Algorithm~\ref{F2} for the last two
examples. The choice of initial value $x_0$ that is more
targeted in Algorithm~\ref{F1} seems to make the difference.



\end{document}